\documentclass[pdflatex,sn-mathphys-numt]{sn-jnl}

\usepackage{amsmath,amsthm,amsfonts}

\usepackage{graphicx}
\usepackage{multirow}%
\usepackage{amssymb}
\usepackage{mathrsfs}
\usepackage[title]{appendix}%
\usepackage{xcolor}%
\usepackage{textcomp}%
\usepackage{manyfoot}%
\usepackage{booktabs}%
\usepackage{algorithm}
\usepackage{algorithmicx}%
\usepackage{algpseudocode}%
\usepackage{listings}%

\usepackage{enumerate}
\usepackage{float}
\usepackage{hyperref} 
\usepackage{dsfont}
\usepackage{amscd}
\usepackage{latexsym}
\usepackage{graphics}
\usepackage{color}

\usepackage{subfigure}

\usepackage{tikz}
\usetikzlibrary{automata, positioning}
\usetikzlibrary{matrix}
\usepackage{pgfplots}

\newtheorem{theorem}{Theorem}[section]
\newtheorem{lemma}[theorem]{Lemma}

\theoremstyle{definition}
\newtheorem{definition}[theorem]{Definition}

\theoremstyle{remark}
\newtheorem{remark}[theorem]{Remark}

\numberwithin{equation}{section}

\newcommand{\cH}{\mathcal{H}}

\newcommand{\cP}{\mathcal{P}}

\newcommand{\bP}{\mathbb{P}}
\newcommand{\C}{\mathbb{C}}

\newcommand{\Z}{\mathbb{Z}}

\newcommand{\R}{\mathbb{R}}

\newcommand{\ep}{\epsilon }

\newcommand{\si}{\sigma }

\newcommand{\ga}{\gamma }
\newcommand{\Ga}{\Gamma }

\newcommand{\Om}{\Omega }

\newcommand{\Var}{\operatorname{Var}}

\newcommand{\cReff}{\eff\mathcal{R}}

\newcommand{\eff}{\textrm{eff}}

\newcommand{\UST}{\operatorname{UST}}

\pgfplotsset{soldot/.style={color=black,only marks,mark=*}}

\newcommand{\bi}{\begin{itemize}}
\newcommand{\ei}{\end{itemize}}

\begin{document}

\title[Article Title]{The scaling limit of fair Peano paths}

\author[1]{\fnm{Nathan} \sur{Albin}}

\author*[2]{\fnm{Joan} \sur{Lind}}\email{jlind@utk.edu}

\author[1]{\fnm{Pietro} \sur{Poggi-Corradini}}

\affil[1]{\orgdiv{Department of Mathematics}, \orgname{Kansas State University}, \city{Manhattan}, \state{KS}, \postcode{66506}}

\affil[2]{\orgdiv{Department of Mathematics}, \orgname{University of Tennessee}, \city{Knoxville}, \state{TN},  \postcode{37996}}

\abstract{ We study random Peano paths on planar square grids that arise from fair random spanning trees. These are trees that are sampled in such a way as to have the same (if possible) edge probabilities.
In particular, we are interested in identifying the scaling limit as the mesh-size of the grid tends to zero. 
It is known \cite{lawler-schramm-werner2002} 
that if the trees are sampled uniformly, then the scaling limit exists and equals ${\rm SLE}_8$. We show that if we simply follow the same steps as in 
\cite{lawler-schramm-werner2002}, 
then fair Peano paths have a deterministic scaling limit.}



\maketitle

\section*{Acknowledgments}
The first and last authors acknowledge support from National Science Foundation under Grant  No.~2154032.
This work was partially completed while the second and third author were in residence at the Mathematical Sciences Research Institute in Berkeley, California, and the authors gratefully acknowledge the support provided by MSRI and 
the NSF under Grant No.~1440140.

\large

\section{Introduction}
It was shown in \cite{lawler-schramm-werner2002} that the scaling limit of  random Peano paths arising from uniform spanning trees of planar grids exists and equals  ${\rm SLE}_8$, which is a Schramm-Loewner evolution process. In this paper, we study random Peano paths which are generated by random spanning trees that are not necessarily uniform. In particular, we are interested in studying the limiting behavior of laws on spanning trees that arise in the context of spanning tree modulus. 

Spanning tree modulus is a particular case of the more general theory of modulus of families of objects on graphs. A comprehensive development of the theory of modulus and its specializations can be found in~\cite{albin2025mathematics}. Briefly, discrete modulus is a graph-theoretic analogue of the classical theory of modulus of curve families in the plane. At its core, modulus can be thought of as a measure of the richness of a family of objects on a graph, analogously to the continuum theory: families of many short objects generally have larger modulus than families of few longer objects. When the family of objects in question is the family of spanning trees, the modulus problem has a few interesting equivalent formulations. Most importantly for the present paper, is the Fairest Edge Usage (FEU) problem introduced in Section~\ref{sec:feu}. The problem is essentially to choose a probability distribution on the set of spanning trees so that the edge usage probabilities are as close to equal as possible. An interesting structural property of the FEU problem is that, while the optimal choice of probility distribution is not generally unique, there are typically certain trees that can never occur in the support of any optimal pmf. These are the \emph{forbidden trees}, in the language of modulus. The remaining trees (those that do occur in the support of at least one optimal probability distribution) are called \emph{fair trees}. These are the trees of interest in the present paper. We refer the reader to \cite{achpcst,alpc} for a more detailed discussion of fair trees and their properties.

In order to state our main result, we must first set the stage.
Let $Q=(0,1] \times [0,1)$ be the unit square (including the bottom and right sides).  Note that we will frequently consider $Q$ and other sets as subsets of $\mathbb{C}$.  
For the lattice $\mathcal{L}_n =  \left[\frac{1}{n}\mathbb{Z} \times \frac{1}{n}\mathbb{Z}\right] \cap Q$,
consider the spanning trees that contain the edges along the bottom and right sides of $Q$.
Let $\Gamma_n$ be the set of fair trees, which are identified in Lemma \ref{lem:fairtree-modgrid} (after contracting the bottom and right sides). 
For each element $\gamma \in \Gamma_n$,
there is a unique spanning tree $\hat{\gamma}$ of the dual grid (i.e. $ \left[(\frac{1}{2n}+\frac{1}{n}\mathbb{Z}) \times (\frac{1}{2n}+ \frac{1}{n}\mathbb{Z})\right] \cap Q$) so that $\gamma \cap \hat{\gamma} = \emptyset$.
Further we can define a path $\zeta_\gamma$ in $Q$ from 0 to $1+i$ that winds between $\gamma$ and $\hat{\gamma}$, as illustrated in Figure \ref{TreeAndDual}.  
We take $\zeta_\gamma$ to be on the lattice $ (\frac{1}{4n} + \frac{1}{2n}\mathbb{Z}) \times (\frac{1}{4n} +\frac{1}{2n}\mathbb{Z})$, 
except for near 0 and $1+i$, where we have a short line segment with slope 1.
We call $\zeta_\gamma$ the {\it Peano path} associated with $\gamma$, 
and we define $\mathcal{A}_n = \{ \zeta_\gamma \, : \, \gamma \in \Gamma_n \}$
to be the collection of fair Peano paths.

 \begin{figure}
 \centering
\begin{tikzpicture}[scale=0.9,, auto, node distance=3cm,  thin]
   \begin{scope}[every node/.style={circle,draw=black,fill=black!100!,font=\sffamily\Large\bfseries}]
    \node (v0) [scale=0.2] at (0,0) {};
   \node (v1)[scale=0.2] at (1,0) {};
   \node (v2) [scale=0.2]at (2,0) {};
    \node (v3)[scale=0.2] at (3,0) {};
    \node (v4)[scale=0.2] at (4,0) {};
    \node (v5)[scale=0.2] at (5,0) {};
   \node (w1)[scale=0.2] at (1,1) {};
   \node (w2) [scale=0.2]at (2,1) {};
    \node (w3)[scale=0.2] at (3,1) {};
    \node (w4)[scale=0.2] at (4,1) {};
    \node (w5)[scale=0.2] at (5,1) {};
   \node (x1)[scale=0.2] at (1,2) {};
   \node (x2) [scale=0.2]at (2,2) {};
    \node (x3)[scale=0.2] at (3,2) {};
    \node (x4)[scale=0.2] at (4,2) {};
    \node (x5)[scale=0.2] at (5,2) {};
   \node (y1)[scale=0.2] at (1,3) {};
   \node (y2) [scale=0.2]at (2,3) {};
    \node (y3)[scale=0.2] at (3,3) {};
    \node (y4)[scale=0.2] at (4,3) {};
    \node (y5)[scale=0.2] at (5,3) {};
   \node (z1)[scale=0.2] at (1,4) {};
   \node (z2) [scale=0.2]at (2,4) {};
    \node (z3)[scale=0.2] at (3,4) {};
    \node (z4)[scale=0.2] at (4,4) {};
    \node (z5)[scale=0.2] at (5,4) {};
    \node (vf)[scale=0.2] at (5,5) {};
       \end{scope}
   \begin{scope}[every edge/.style={draw=black,thin}]
    \draw  (v0) edge node{} (v5);
    \draw  (v5) edge node{} (vf);
     \draw  (v1) edge node{} (w1);
     \draw  (v2) edge node{} (x2);
    \draw  (v4) edge node{} (w4);
     \draw  (w3) edge node{} (w4);
   \draw  (x1) edge node{} (x2);
    \draw  (x3) edge node{} (x5);
    \draw  (x3) edge node{} (z3);
     \draw  (y1) edge node{} (y3);
     \draw  (y2) edge node{} (z2);
     \draw  (z1) edge node{} (z2);
     \draw  (x4) edge node{} (y4);
     \draw  (z4) edge node{} (z5);
            \end{scope}
   \begin{scope}[every node/.style={circle,draw=gray,fill=gray!100!,font=\sffamily\Large\bfseries}]
   \node (a1)[scale=0.2] at (0.5,0.5) {};
   \node (a2) [scale=0.2]at (1.5,0.5) {};
    \node (a3)[scale=0.2] at (2.5,0.5) {};
    \node (a4)[scale=0.2] at (3.5,0.5) {};
    \node (a5)[scale=0.2] at (4.5,0.5) {};
   \node (b1)[scale=0.2] at (0.5,1.5) {};
   \node (b2) [scale=0.2]at (1.5,1.5) {};
    \node (b3)[scale=0.2] at (2.5,1.5) {};
    \node (b4)[scale=0.2] at (3.5,1.5) {};
    \node (b5)[scale=0.2] at (4.5,1.5) {};
   \node (c1)[scale=0.2] at (0.5,2.5) {};
   \node (c2) [scale=0.2]at (1.5,2.5) {};
    \node (c3)[scale=0.2] at (2.5,2.5) {};
    \node (c4)[scale=0.2] at (3.5,2.5) {};
    \node (c5)[scale=0.2] at (4.5,2.5) {};
   \node (d1)[scale=0.2] at (0.5,3.5) {};
   \node (d2) [scale=0.2]at (1.5,3.5) {};
    \node (d3)[scale=0.2] at (2.5,3.5) {};
    \node (d4)[scale=0.2] at (3.5,3.5) {};
    \node (d5)[scale=0.2] at (4.5,3.5) {};
   \node (e1)[scale=0.2] at (0.5,4.5) {};
   \node (e2) [scale=0.2]at (1.5,4.5) {};
    \node (e3)[scale=0.2] at (2.5,4.5) {};
    \node (e4)[scale=0.2] at (3.5,4.5) {};
    \node (e5)[scale=0.2] at (4.5,4.5) {};
       \end{scope}
 \begin{scope}[every edge/.style={draw=gray,thin}]
    \draw  (a1) edge node{} (e1);
    \draw  (e1) edge node{} (e5);
     \draw  (a2) edge node{} (b2);
     \draw  (a3) edge node{} (c3);
    \draw  (a3) edge node{} (a4);
   \draw  (a5) edge node{} (b5);
    \draw  (b1) edge node{} (b2);
    \draw  (b3) edge node{} (b5);
     \draw  (c1) edge node{} (c3);
     \draw  (c4) edge node{} (e4);
     \draw  (c5) edge node{} (d5);
     \draw  (d1) edge node{} (d2);
     \draw  (d3) edge node{} (e3);
     \draw  (d4) edge node{} (d5);
            \end{scope}
 \begin{scope}[{draw=red, thick}]            
       \draw[-] (0,0) to (0.25,0.25);   
       \draw[-] (0.25,0.25) to (0.75, 0.25); 
       \draw[-] (0.75, 0.25) to (0.75, 1.25); 
        \draw[-] (0.75, 1.25) to (1.25, 1.25); 
        \draw[-] (1.25, 1.25) to (1.25, 0.25); 
         \draw[-] (1.25, 0.25) to (1.75, 0.25); 
          \draw[-] (1.75, 0.25) to (1.75, 1.75); 
           \draw[-] (1.75, 1.75) to (0.75, 1.75); 
            \draw[-] (0.75, 1.75) to (0.75, 2.25); 
             \draw[-] (0.75, 2.25) to (2.25, 2.25); 
             \draw[-] (2.25, 2.25) to (2.25, 0.25); 
             \draw[-] (2.25, 0.25) to (3.75, 0.25); 
             \draw[-] (3.75, 0.25) to (3.75, 0.75); 
             \draw[-] (3.75, 0.75) to (2.75, 0.75); 
             \draw[-] (2.75, 0.75) to (2.75, 1.25); 
             \draw[-] (2.75, 1.25) to (4.25, 1.25); 
             \draw[-] (4.25, 1.25) to (4.25, 0.25); 
             \draw[-] (4.25, 0.25) to (4.75, 0.25);
             \draw[-] (4.75, 0.25) to (4.75, 1.75); 
             \draw[-] (4.75, 1.75) to (2.75, 1.75);
             \draw[-] (2.75, 1.75) to (2.75, 2.75);
             \draw[-] (2.75, 2.75) to (0.75, 2.75);
             \draw[-] (0.75, 2.75) to (0.75, 3.25);
             \draw[-] (0.75, 3.25) to (1.75, 3.25);
             \draw[-] (1.75, 3.25) to (1.75, 3.75);
             \draw[-] (1.75, 3.75) to (.75, 3.75);
             \draw[-] (.75, 3.75) to (.75, 4.25);
             \draw[-] (.75, 4.25) to (2.25, 4.25);
             \draw[-] (2.25, 4.25) to (2.25,3.25);
             \draw[-] (2.25, 3.25) to (2.75, 3.25);
             \draw[-] (2.75, 3.25) to (2.75, 4.25);
             \draw[-] (2.75, 4.25) to (3.25, 4.25);
             \draw[-] (3.25, 4.25) to (3.25, 2.25);
             \draw[-] (3.25, 2.25) to (3.75, 2.25);
             \draw[-] (3.75, 2.25) to (3.75, 3.25);
             \draw[-] (3.75, 3.25) to (4.25, 3.25);
             \draw[-] (4.25, 3.25) to (4.25, 2.25);
             \draw[-] (4.25, 2.25) to (4.75, 2.25);
             \draw[-] (4.75, 2.25) to (4.75, 3.75);
             \draw[-] (4.75, 3.75) to (3.75, 3.75);
             \draw[-] (3.75, 3.75) to (3.75, 4.25);
             \draw[-] (3.75, 4.25) to (4.75, 4.25);
             \draw[-] (4.75, 4.25) to (4.75, 4.75);
             \draw[-] (4.75, 4.75) to (5, 5);
       \end{scope}  
\end{tikzpicture}
\caption{An example of spanning trees $\gamma$ (in black) and $\hat{\gamma}$ (in gray) and the corresponding Peano path $\zeta_\gamma$ (in red) for $n=5$. }\label{TreeAndDual}
\end{figure}
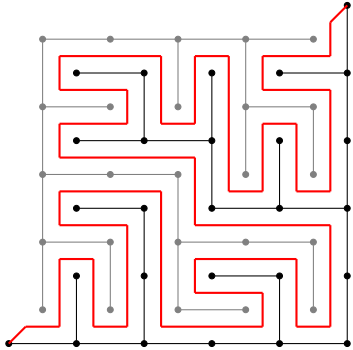

We consider an element $\zeta$ of $\mathcal{A}_n$ to be a family of the growing traces $\zeta[0,t]$ for $t\in [0,1]$ modulo reparametrization.  For $\zeta, \xi \in \cup \mathcal{A}_n$, define the distance 
\begin{equation}\label{distance}
d(\zeta, \xi) = \inf \sup_{t \in [0,1]} d_H\left( \zeta[0,t], \xi[0,t]\right),
\end{equation}
where the infimum is over all possible parametrizations of $\zeta, \xi,$ and $d_H$ is the Hausdorff distance (which is the smallest $\epsilon$ so that every point in each set is within distance $\epsilon$ of the other set).
Let $X= \overline {\cup \mathcal{A}_n}$ be the closure of $\cup \mathcal{A}_n$ with respect to the distance $d$.
Also, let $\cP(X)$ be the set of probability measures on $(X, \mathcal{B}(X))$, 
where $\mathcal{B}(X)$ is the Borel $\sigma$-algebra induced by the metric $d$ on $X$.

Note that elements of $X$ are growing families of subsets of $\overline{Q}$, modulo reparametrization (i.e. elements have the form $\kappa = \{\kappa_t\}_{t \in [0,1]}$ with $\kappa_s \subseteq \kappa_t$ for $s \leq t$).
Further, $X$ contains elements that are no longer traces of curves.  
For instance, 
define the growing family of polygons
\begin{equation}\label{lambda}
\lambda_t := \{ x+iy \in \overline{Q} \, : \, x+y \leq 2t \}.
\end{equation}
We will see in Section \ref{scalinglimit} that $\lambda = \{\lambda_t\}_{t\in[0,1]}$ is an element of $X$. Moreover, the following is our main result. 
\begin{theorem} \label{weakconvergence}
Let $\mu_n \in \cP(X)$ be the probability measure that is uniform on $\mathcal{A}_n$.
Then  $\mu_n$ converges weakly to the Dirac measure $\delta_\lambda$ as $n \to \infty$, where $\lambda   \in X$ is defined by \eqref{lambda}.
\end{theorem}

Note in particular that we obtain a deterministic scaling limit.  
There are other possible choices for fair probability measures.  However, we choose  to study $\mu_n$ since it has the largest support and the largest entropy out of all these possibilities (see Remark \ref{rmk: entropy}).  
In this setting, one can relate the fair trees  to a system of coalescing random walks under isotropic scaling.  
(This is in contrast to the diffusive scaling that would lead to the Brownian web, which originated with Arratia \cite{arratia79}.  See the survey paper \cite{schertzer_sun_swart_2016} for more on the Brownian web and related objects.)

In Section \ref{background}, we summarize the background regarding fair trees that is relevant for our purposes. 
In Section \ref{scalinglimit}, we prove Theorem \ref{weakconvergence} in two steps: we first show the existence of subsequential limits of $\mu_n$, using Prohorov's Theorem, and then we show that the only possible subsequential limit is $\delta_\lambda$.
Finally in Section \ref{sec:fairtrees-grids} we identify the fair trees in the case of both the modified and standard grids, and we give  evidence supporting an alternate approach.

\section{Background on fair trees and modified grids}\label{background}
\subsection{Random spanning trees}
Consider a connected graph $G=(V,E)$, with vertex-set $V$ and edge-set $E$ that are both finite. We allow multi-edges, but not self-loops. Let $\Ga_G$ be the set of all spanning trees of $G$. The set of all probability mass functions (pmf's) $\mu$ on $\Ga_G$ is denoted with $\cP(\Ga_G)$. A pmf $\mu\in\cP(\Ga_G)$ determines  random trees $\underline{\ga}$ that satisfy the law:
\[
\bP_\mu\left(\underline{\ga}=\ga\right)=\mu(\ga)\qquad\forall \ga\in\Ga_G.
\]

The {\it $\si$-weighted uniform spanning trees}, written  $\UST_\si$ for short, are random spanning trees obtained by considering some fixed edge-weights $\si\in\R_{>0}^E$. 
These are random trees $\underline{\ga}\in\Ga_G$ whose probability $\mu_\si(\ga)$ is proportional to  the product:
\[
\prod_{e\in \ga} \si(e).
\]
If the edge weights $\si$ are constant, then we recover uniform spanning trees and simply write $\UST$.
It is well known, see \cite[Chapter 4]{LP:book}, that $\UST_\si$ can be sampled using weighted random walks as in the Aldous-Broder algorithm, or  loop-erased weighted random walks as in Wilson's algorithm. 
The basic idea of these algorithms is to grow a spanning tree by adding new edges from an appropriate random walk, and we will use this idea in our work.

Lastly, we mention a celebrated result of Kirchhoff which relates the edge-probabilities for $\UST_\si$ and the effective resistance $\cReff_\si$ for the associated electrical network with edge-conductances given by $\si$. More precisely, if $e\in E$, then
\begin{equation}\label{eq:peredgeeffres}
\bP_{\mu_\si}\left(e\in \underline{\ga}\right) = \si(e)\cReff_\si(e).
\end{equation}
See   ``Kirchhoff's Effective Resistance Formula" in \cite[Chapter 4]{LP:book}.
\subsection{Fair edge usage}\label{sec:feu}
As we have seen, the edge-probabilities of UST coincide with the effective resistance of each edge and can therefore vary from edge to edge. Instead of requiring that each spanning tree be equally likely as in the case of UST, we ask for laws that sample spanning trees so as to equalize the edge-probabilities as much as possible. Ideally, we would like all the edge-probabilities to be the same, but this is not always possible. When this happens we say the graph is {\it homogenous}. More precisely, the graph $G$ is said to be homogeneous, if it admits a pmf $\mu^*\in\cP(\Ga_G)$ with constant edge probabilities, in which case  by \cite[Corollary 4.5]{achpcst}, the edge-probabilities are uniquely determined as
\begin{equation}\label{eq:edge-prob}
\bP_{\mu^*}\left(e\in\underline{\ga}\right)\equiv\frac{|V|-1}{|E|},\qquad\forall e\in E.
\end{equation}
In general, the Fairest Edge Usage  (FEU) problem, introduced in \cite{achpcst}, tries to minimize the variance of the edge-probabilities:
\begin{equation}\label{prob:feu}
\begin{array}{lll}
\underset{\mu\in\cP(\Ga_G)}{\text{minimize}} && \Var (\eta) \\
&&\\
\text{subject to} &&  \eta(e)=\bP_\mu\left(e\in\underline{\ga}\right), \quad \forall e\in E.
\end{array}
\end{equation}
A pmf $\mu$ that is optimal for the FEU problem (\ref{prob:feu}) is called  {\it fair}, and a spanning tree $\gamma$ is  {\it fair} if it is in the support of some fair pmf.
In \cite{achpcst} it was shown that fair pmf's for FEU always exist, but they are in general not unique. On the other hand, the edge probabilities $\eta^*(e):=\bP_{\mu^*}\left(e\in\underline{\ga}\right)$ corresponding to a fair pmf $\mu^*$, are uniquely determined. In fact, the optimal $\eta^*$ in (\ref{prob:feu}) can be computed using the spanning tree modulus algorithms and is related to the family of feasible partitions via Fulkerson duality. More details about these connections can be found in \cite{achpcst}.

We write $\Ga_G^f$ for the set of all fair trees. Trees in $\Ga_G\setminus\Ga_G^f$ are called {\it forbidden} trees. In particular, if a $\UST_\si$ pmf $\mu_\sigma$ is fair, then $G$ does not have any forbidden trees, because $\mu_\sigma$ is supported on all of $\Gamma_G$ by definition.
\begin{definition}\label{def:allfair}
A graph $H$ has the  {\it all fair} property if every spanning tree of $H$ is fair: $\Ga_H=\Ga_H^f$.
\end{definition}
Graphs that have the all fair property are characterized in Theorem \ref{thm:main} below, and as we will prove in Lemma \ref{lem:planar-grids}, planar square grids have this property as well.

\subsection{Planar modified grids}\label{sec:modified-grid}

The scaling limit for UST was established in \cite{lawler-schramm-werner2002}.  The idea there is to consider a modified grid, see Figure \ref{fig:modified-grid}, so that each spanning tree gives rise to a normalized Peano path. The authors in \cite{lawler-schramm-werner2002}  show that as the mesh size of the grid goes to zero, these random Peano paths converge to a known stochastic process called $SLE_8$.

A {\it standard $m$-by-$k$ planar grid} is a square grid with vertex set
$ \left(\mathbb{Z} \times \mathbb{Z}\right) \cap \left([0,  m] \times [0,  k] \right).$
A modified grid $G$ consists of a standard $m$-by-$k$ grid $G_0$, together with an extra node $v_0$ connected via an edge to each node on the bottom and the right-hand side of $G_0$. In particular, the bottom-right corner of $G_0$ is connected by two edges to $v_0$. Alternatively, we can think of $G$ as the graph obtained from an $(m+1)$-by-$(k+1)$ grid by identifying every node on the bottom and right-hand sides and removing all the resulting self-loops.

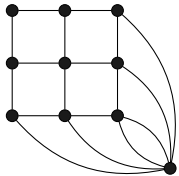
\begin{figure}
\centering
\begin{tikzpicture}[scale=0.7]
\begin{scope}[every node/.style={circle,draw=black,thin,fill=black!90! ,}]
 \node  (v1)[scale=0.4] at (-2,-2){};
    \node (v2)[scale=0.4] at (-1,-2){};
    \node (v3) [scale=0.4] at (0,-2){};
    \node (v4) [scale=0.4] at (-2,-1){};
    \node (v5) [scale=0.4] at (-1,-1){};
   \node (v6) [scale=0.4] at (0,-1){};
 \node (v7) [scale=0.4] at (-2,0) {};
   \node (v8) [scale=0.4] at (-1,0){};
   \node (v9) [scale=0.4] at (0,0) {} ;
   \node (v10)[scale=0.4] at (1,-3){};
  \end{scope}
 
\begin{scope}[every edge/.style={draw=black, thin}]
    \draw  (v1) edge node{} (v2);
    \draw  (v1) edge node{} (v4);

    \draw (v2) edge node{} (v3);
      \draw (v2) edge node{} (v5);

     \draw (v3) edge node{} (v6);

   \draw (v6) edge node{} (v5);

    \draw (v5) edge node{} (v4);
      
      \draw (v4) edge node{} (v7);

      \draw (v7) edge node{} (v8);  
      \draw (v8) edge node{} (v9);
  
\draw (v9) edge node{}(v6);
\draw (v8) edge node{} (v5);

\path (v1) edge [bend right] node{}(v10);
  \path (v2)  edge [bend right] node{}(v10);
  
  \path (v3)  edge [bend right] node{}(v10);

  \path (v3)  edge [bend left] node{}(v10);
  \path (v6) edge[bend left] node{} (v10);
  \path (v9) edge [bend left] node{} (v10);
 \end{scope}
\end{tikzpicture}  
\caption{Modified Grid}
\label{fig:modified-grid}
 \end{figure}
We associate to every vertex $v \in V(G_0)=V(G)\setminus\{v_0\}$ in the grid  an edge set $E_v\subset E(G)$, consisting of the edge connecting $v$ to its neighbor on the right and the one below $v$. Note that if $v$ belongs to the bottom or the right-hand side edge of the grid $G_0$, then the neighbor to the right, or the one below $v$, may have to be the special node $v_0$. Then, the family of sets $\{E_v\}_{v\in G_0}$ forms a partition of $E$:
\[
E(G)=\bigcup_{v\in V(G)\setminus\{v_0\}}E_v\qquad\text{and}\qquad E_v\cap E_{v'}=\emptyset,\quad\text{for $v\neq v'$}.
\]
 Consider the collection of all subgraphs of $G$ that use exactly one edge from each $E_v$:
\begin{equation}\label{eq:forbidden-turns}
\bar{\Gamma}:= \{T \subset E:\ |T \cap E_v|=1, \quad \text{for all $v \in V(G_0)$} \}.
\end{equation}
For instance, the tree on the left in Figure \ref{fig:grid-trees} is in $\bar{\Ga}$, while the tree on the right is not, because the latter tree connects the  central node of the grid to both its neighbor on the right and its neighbor below.
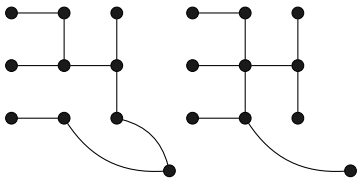
\begin{figure} [h!]
\centering
\begin{tikzpicture}[scale=0.7]
 \begin{scope}[ every node/.style={circle,draw=black,fill=black!90!,font=\sffamily\small\bfseries}]
    \node (v1) [scale=0.4]at (-2,-2) {};
    \node (v2) [scale=0.4] at (-1,-2) {};
    \node (v3) [scale=0.4] at (0,-2) {} ;
    \node (v4) [scale=0.4] at (-2,-1){};
   \node (v5)  [scale=0.4] at (-1,-1) {} ;
   \node (v6) [scale=0.4] at (0,-1){};
   \node (v7) [scale=0.4] at (-2,0) {};
   \node (v8) [scale=0.4]at (-1,0) {};
   \node (v9) [scale=0.4] at (0,0) {};
   \node (v10) [scale=0.4] at (1,-3) {};
   \end{scope}
   \begin{scope}[every edge/.style={draw=black, thin}]
   \draw (v1) edge node{} (v2);
   \path (v2) edge[bend right] (v10);
   \path (v3) edge[bend left] (v10);
   \draw (v4) edge node{} (v5);
   \draw (v6) edge node{} (v5);
   \draw (v3) edge node{} (v6);
   \draw (v6) edge node{} (v9);
   \draw (v7) edge node{} (v8);
   \draw (v8) edge node{} (v5);
   \end{scope}
   \end{tikzpicture}
\begin{tikzpicture}[scale=0.7]
\begin{scope}[ every node/.style={circle,draw=black,fill=black!90!,font=\sffamily\small\bfseries}]
    \node (v1) [scale=0.4]at (-2,-2) {};
    \node (v2) [scale=0.4] at (-1,-2) {};
    \node (v3) [scale=0.4] at (0,-2) {} ;
    \node (v4) [scale=0.4] at (-2,-1){};
  \node (v5) [scale=0.4] at (-1,-1) {} ;
   \node (v6)[scale=0.4] at (0,-1){};
 \node (v7)[scale=0.4] at (-2,0) {};
   \node (v8) [scale=0.4]at (-1,0) {};
   \node (v9)[scale=0.4] at (0,0) {};
   \node (v10) [scale=0.4] at (1,-3) {};
  
   \end{scope}
   \begin{scope}[every edge/.style={draw=black, thin}]
   \draw (v1) edge node{} (v2);
   \path (v2) edge[bend right] (v10);
   \draw (v2) edge node{} (v5);
   \draw (v4) edge node{} (v5);
   \draw (v6) edge node{} (v5);
   \draw (v3) edge node{} (v6);
   \draw (v6) edge node{} (v9);
   \draw (v7) edge node{} (v8);
   \draw (v8) edge node{} (v5);
   \end{scope}
   \end{tikzpicture}
  
\caption{Fair tree on the left and Forbidden tree on the right}
 \label{fig:grid-trees}
 \end{figure}

Suppose we are given a set $T\in\bar{\Ga}$. Then, starting from any node $v\in V(G_0)$ one can follow an edge of $T$ and either move to the right or down and eventually reach $v_0$. Thus, every $T \in \bar{\Gamma}$ is connected and spans $V(G)$.
Moreover, any such $T$ satisfies
\begin{equation} \label{eq:edgecount}
|E(T)|=\sum_{v \in V(G_0)} |E_{v} \cap T|=\sum_{v \in V(G_0)} 1= |V(T)|-1.
\end{equation} 
Therefore, every such $T$ is a spanning tree of $G$, namely $\bar{\Gamma} \subset \Gamma_G$

Let $\bar{\mu}$ be the uniform pmf conditioned  on $\bar{\Gamma}$:
\[ \bar{\mu}(T):= \begin{cases} 
      \frac{1}{\left|\bar{\Gamma}\right|} &  \text{if $T \in \bar{\Gamma}$} \\
      0 & \text{otherwise}.\\
      
   \end{cases}
\]
A random tree in $\bar{\Gamma}$ with law $\bar{\mu}$ can be sampled by tossing independent  fair coins, one for each node  $v\in V(G_0)$, and picking either the edge to the right of $v$ or below $v$. In particular, notice that every tree $T\in \bar{\Gamma}$ has a  {\it partner tree}, i.e., the tree that can be obtained by substituting  every edge in $T\cap E_v$ with the other edge in $E_v\setminus T$.

\begin{lemma}\label{lem:fairtree-modgrid}
The family $\bar{\Ga}$ is the family of all fair trees for the modified grid, i.e., 
$\Gamma^f_G=\bar{\Gamma}$. Also,  $\bar{\mu}$  is optimal for the $\text{FEU}$ problem on $G$.

Moreover, $\bar{\Ga}$ consists exactly of all the trees $\ga\in\Gamma_G$ that have the {\it partner tree property}, i.e., the trees $\ga\in\Ga_G$ such that the complement $E\setminus \ga$ is also a tree.
\end{lemma}
For the moment, we will assume Lemma \ref{lem:fairtree-modgrid}, and defer its proof to Section \ref{sec:modgrids}.

\begin{remark} \label{rmk: entropy}
Since every fair pmf must have its support in the set $\Gamma^f_G$ of all fair trees, Lemma \ref{lem:fairtree-modgrid} shows that $\bar{\mu}$ has  maximal support for a fair pmf.  
Recall that if we compare the entropy of pmfs supported on $\Gamma^f_G$, then the uniform pmf will maximize entropy.  Thus  $\bar{\mu}$ has the maximal entropy for a fair pmf.
\end{remark}

\section{The scaling limit of fair Peano paths} \label{scalinglimit}

We are interested in the behavior of fair Peano paths as $n \to \infty$.  
Since any fair Peano path $\zeta_n \in \mathcal{A}_n$ passes within distance $\frac{1}{4n}$ of every vertex in the lattice $\mathcal{L}_n$, the  limit of the full trace $\zeta_n[0,1]$ will converge to $\overline{Q}$.
However, we are interested in a more detailed understanding.  In particular, is the limit a curve or not, and is the limit a random or deterministic object?
Theorem \ref{weakconvergence}  answers both of these questions, showing that the limit is not a curve and is deterministic.

\begin{figure}
 \centering
\includegraphics[scale=0.25]{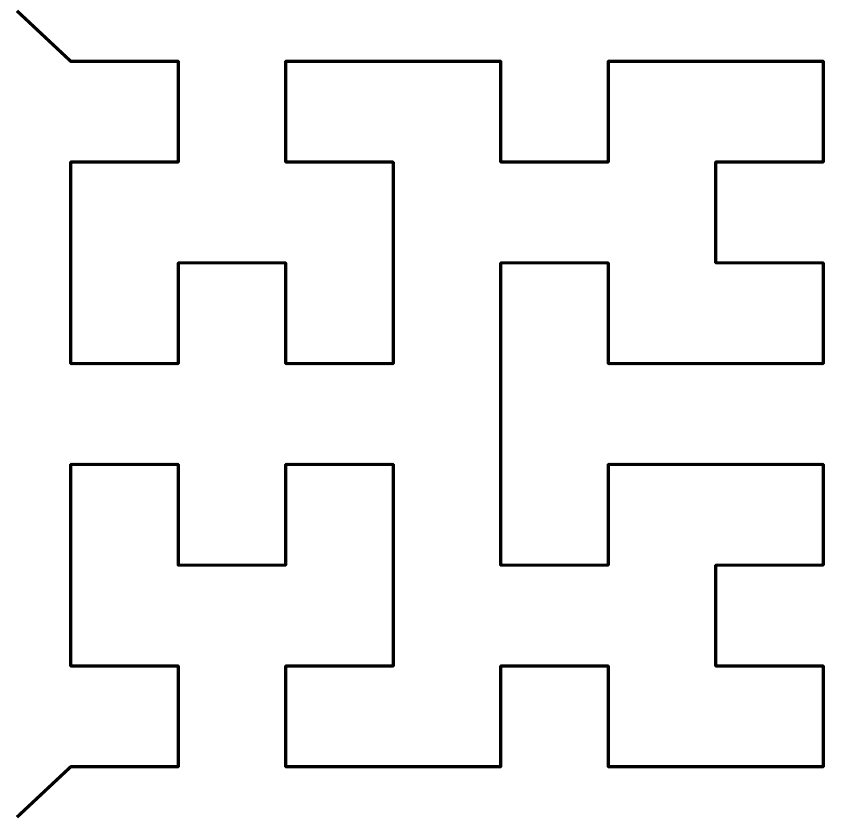}  \hspace{0.25in}
\includegraphics[scale=0.25]{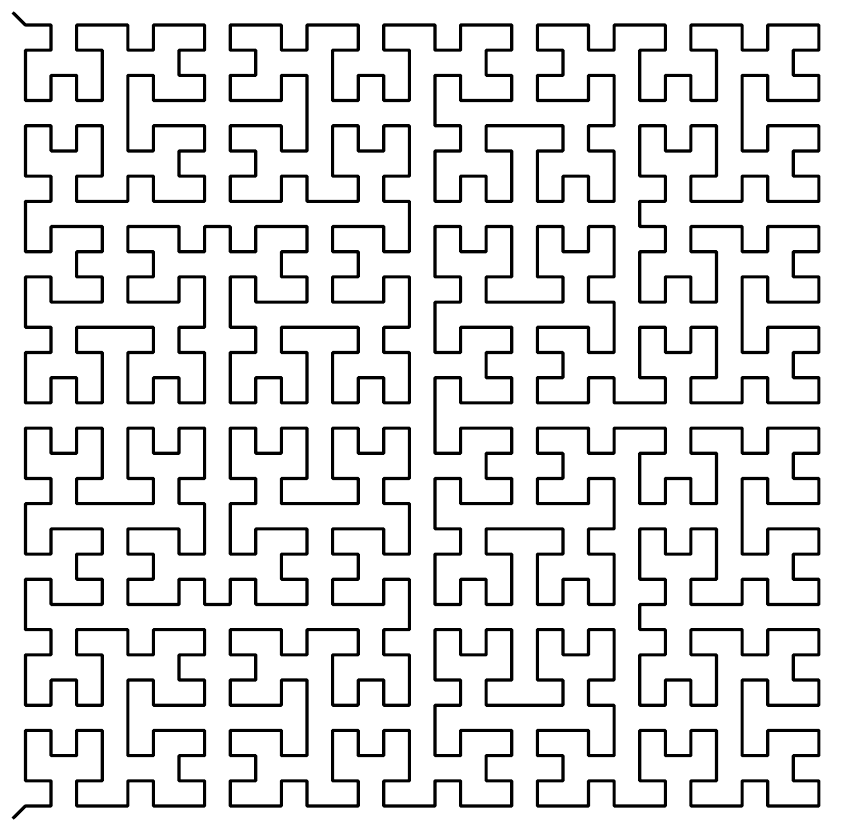}  \hspace{0.25in}
\includegraphics[scale=0.25]{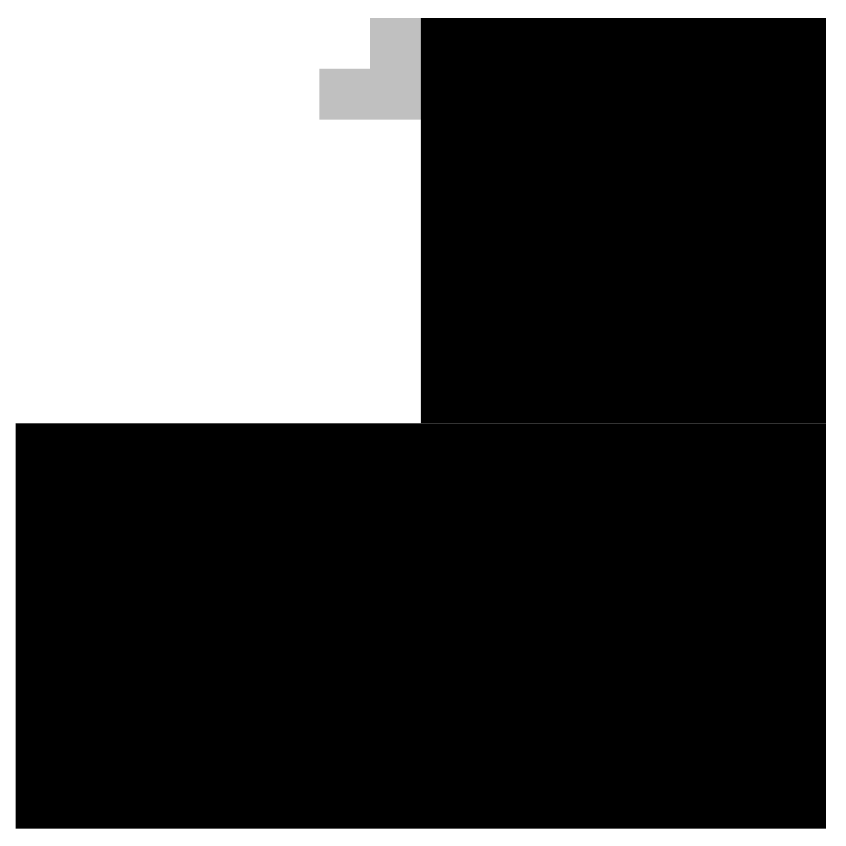} 
\caption{The Hilbert spacefilling curve is a continuous function $h$ from $[0,1]$ onto $\overline{Q}$.  Two approximating curves  are shown (left and middle).  The right image shows $h[0, 3/4]$ in black and $h[3/4, 3/4+\epsilon]$ for $\epsilon = 3/64$ in grey.
}\label{Hilbert}
\end{figure}

Before getting to the proof, let us pause for a moment to briefly discuss the notion of spacefilling curves.  
A {\it spacefilling curve} in $\overline{Q}$ is a continuous function $\phi: [0,T] \to \overline{Q}$ so that $\phi[0,T] = \overline{Q}$.
Since the image of the path fills up the entire square, a spacefilling curve is best visualized by an approximation.  For an example, see Figure \ref{Hilbert} which shows two curves that approximate the Hilbert spacefilling curve.  
We can view a spacefilling curve as a way to grow $\overline{Q}$ in a continuous fashion.  The continuity implies that we can always make the diameter of  $\phi[t_0, t_0+\epsilon]$ small by taking $\epsilon$ small enough. 
Note that $\lambda$ defined by \eqref{lambda} does not arise from a spacefilling curve, since we cannot make the diameter of $\lambda_{t_0+\epsilon} \setminus \lambda_{t_0}$ small by taking $\epsilon$ small.

Next we wish to briefly explain why $\lambda$  is in $X$, which is the closure of $\cup \mathcal{A}_n$ with respect to the distance $d$ in  (\ref{distance}).
Create a tree $\gamma$, so that when we remove the edges along the bottom and right sides, then the remainder is a collection of  ``staircase" paths that alternate between horizontal and vertical edges.  
Note that by Lemma \ref{lem:fairtree-modgrid}, $\gamma \in \Gamma_n$, and so the associated Peano path is in  $\mathcal{A}_n$.
Then $\lambda$ will be the limit of these ``staircase" Peano paths under the distance $d$.
In fact, we will see shortly that fair Peano paths are close to $\lambda$ with high probability under $\mu_n$.

As a first step, we analyze the spanning trees in $\Gamma_n$ and prove a result which roughly says that with high probability, the branches of the trees in $\Gamma_n$ are close to being ``diagonal".
The main idea relies on the fact from
Section \ref{sec:modified-grid} that we can sample a tree in $\Gamma_n$ under the uniform law $\bar{\mu}_{n}$
by tossing independent  fair coins, one for each node  $v$ in the interior of $Q$, to select either the edge to the right of $v$ or below $v$.
(Note that the uniform pmf $\bar{\mu}_{n}$ on fair trees $\Gamma_n$ corresponds to the uniform pmf $\mu_n$ on fair Peano paths $\mathcal{A}_n$ via the isomorphism between Peano paths and trees.)
Now fix a node $z$ in $\mathcal{L}_n$. 
We can create a random walk on $\mathcal{L}_n$ starting at $z$ by iteratively adding edges, so that at each step, we select either the edge to the right or below the current node with equal probability, continuing until we reach  the bottom/right boundary of $Q$.
Notice that the resulting walk  is a rotated/scaled graph of a simple random walk.
Further for a fixed path of this type, the probability under $\bar{\mu}_n$ that a tree in $\Gamma_n$ contains this path  is $2^{-(\text{length of path})}$, which is also the probability that  this path is the beginning of a simple random walk (after appropriate rotation/scaling).

Building on this idea, we want to look at a collection of branches, instead of just one at a time.  
Fix $n \in \mathbb{N}$ and let $m$ be the integer part of $n^{1/4}$.  
Let  $z_1, \cdots, z_m \in \{x+iy : x=1/n \text{ or } y=1-1/n \}$ be spaced out as evenly as possible along the left and top vertices of the lattice $\mathcal{L}_n$.
We assume these points are ordered clockwise, as illustrated in
Figure \ref{BetaAndL}.
For $\zeta \in \mathcal{A}_n$, let $\gamma_\zeta$ be the corresponding tree in $\Gamma_n$, and define the set 
 $\beta = \beta(\zeta) \subset \gamma_\zeta$ to be the union of paths that start at $z_j$ for $j \in \{1, 2, \cdots, m\}$ and take steps either to the right or down following $\gamma_\zeta$ until reaching the bottom or right side of $Q$ 
 (i.e. $\beta$ is a collection of branches of $\gamma_\zeta$).  
Let $L$ be the union of the line segments $l_j =\{ x+iy \in \overline{Q} \, : \, x + y = \text{Re}(z_j) + \text{Im}(z_j) \}$ for $j \in \{1, 2, \cdots, m\}$.
Note that the distance between $l_j$ and $l_{j+1}$ is roughly $\sqrt{2}/m$ with an error bounded by $2\sqrt{2}/n$, and hence it will be  larger than $n^{-1/4}$ but less than $2n^{-1/4}$.

 \begin{figure}
 \centering
\begin{tikzpicture}
\begin{axis}[
xmin=-0.05,xmax=1.1,ymin=0, ymax=1.1,x=2in,y=2in, axis x line*=bottom, axis y line*=left,
axis line style={draw opacity=0}
]
\pgfplotsset{ticks=none}
\addplot[domain=-0.02:0.33] {0.33-x};
\addplot[domain=0:0.35, samples=50,smooth,red] {0.33+0.02*sin(deg(50*x))-x};
\addplot[domain=-0.02:0.67]{0.67-x};
\addplot[domain=0:0.7, samples=50,smooth,red] {0.67-0.022*sin(deg(50*x))-x};
\addplot[domain=-0.02:1]{1-x};
\addplot[domain=0:1, samples=50,smooth,red] {1-0.02*sin(deg(50*x))-x};
\addplot[domain=0.31:1]{1.33-x};
\addplot[domain=0.33:1, samples=50,smooth,red] {1.33+0.02*sin(deg(50*(x-0.33)))-x};
\addplot[domain=0.65:1]{1.67-x};
\addplot[domain=0.67:1, samples=50,smooth,red] {1.67-0.02*sin(deg(50*(x-0.67)))-x};
  \addplot[soldot] coordinates{(0,0.33)};
   \addplot[soldot] coordinates{(0,0.67)};
    \addplot[soldot] coordinates{(0,1)};
     \addplot[soldot] coordinates{(0.33,1)};
      \addplot[soldot] coordinates{(0.67,1)};
\end{axis}
\node[left] at (0.2, 1.66) {$z_1$};
\node[left] at (0.2, 3.33) {$z_2$};
\node[left] at (0, 4.4) {$\vdots$};
\node[left] at (3.9, 5.5) {$z_m$};
\end{tikzpicture}
\caption{A sketch of  $\beta$ (in red), $L$ (in black), and $z_1, \cdots, z_m$ for $n=625, m=5$. }\label{BetaAndL}
\end{figure}
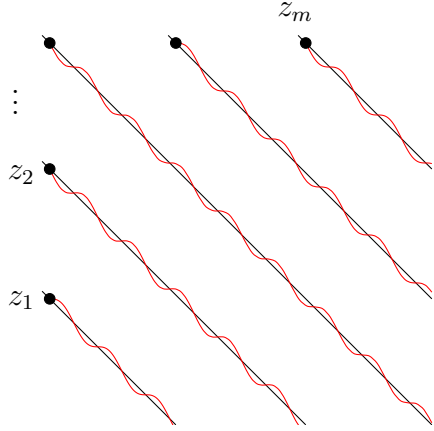

\begin{lemma} \label{lem:diagonal}
Fix $n \in \mathbb{N}$ and let $\beta = \beta(\zeta)$ and $L$ be defined as above.  Then
$$ \mu_n \left( \Big\{ \zeta \in  \mathcal{A}_n  \, : \, d_H\left(\beta(\zeta), \, L \right) < \tfrac{1}{2} n^{-1/4} \Big\} \right) \geq 1-32n^{-1/4},$$
where $d_H$ refers to Hausdorff distance.
\end{lemma}

\begin{proof}
\large
We begin by relating $\beta$ to a simple random walk.  To that end, let $S_k$ be a 1-dimensional simple random walk with step size 1.
We can create $\beta$ from $S_k$ by chopping the walk into $m$ pieces of length $2n$, and then scaling, rotating, shifting, and truncating these pieces.  In particular, fix $j \in \{0, \cdots, m-1\}$ and 
assume that we have created $j$ paths $\alpha_1, \cdots, \alpha_{j}$.  
Then to create $\alpha_{j+1}$ we start at $z_{j+1}$ and follow the path of $\frac{\sqrt{2}}{2n}e^{-\frac{\pi}{4}i}(S_k-S_{2nj}) +z_{j+1}$ for $k \geq 2nj$ until reaching the boundary of $Q$ or until hitting $\cup_{i =1}^j\alpha_{i}$.  
Then
$\beta$ and $ \cup_{i=1}^m \alpha_i$ are equal in distribution by Lemma \ref{lem:fairtree-modgrid}.

 This connection between $\beta$ and the simple random walk implies that
 \begin{align*}
  \mu_n &\left( \Big\{ \zeta \in  \mathcal{A}_n  \, : \, d_H\left(\beta(\zeta), \, L \right) < \tfrac{1}{2} n^{-1/4} \Big\} \right) \\
       &\;\; \geq  \mathbb{P}\left( \tfrac{\sqrt{2}}{2n} |S_{2nj+i} - S_{2nj}| < \tfrac{1}{2} n^{-1/4} \text{ for } j \in \{ 0, \cdots, m-1\}, i \in \{1, \cdots, 2n \}\right). 
  \end{align*}
  Let $M_k = \max\{ |S_0|, |S_1|, \cdots, |S_k|\}$ be the running maximum of $|S_k|$.
 Then since $M_k < r$ implies that $|S_i-S_j| <2 r$ for any $i,j \in \{0, \cdots, k\}$, the equation above is bounded below by       
     $ \mathbb{P}\left(\tfrac{\sqrt{2}}{2n} M_{2mn} < \tfrac{1}{4} n^{-1/4} \right) 
      = \mathbb{P}\left( M_{2mn} < \tfrac{\sqrt{2}}{4} n^{3/4} \right)$.
Therefore,
  \begin{equation*} 
    \mu_n \left( \Big\{ \zeta \in  \mathcal{A}_n  \, : \, d_H\left(\beta(\zeta), \, L \right)      < \tfrac{1}{2} n^{-1/4} \Big\} \right)   
            \geq \mathbb{P}\left( M_{2mn} < \tfrac{\sqrt{2}}{4} n^{3/4} \right) 
            \geq 1-32n^{-1/4},
 \end{equation*}
 since $\mathbb{P}(M_k < r) \geq 1-2k/r^2$ by Lemma \ref{runningmax}.
\end{proof}

\begin{lemma} \label{closetolambda}
Let $\zeta \in  \mathcal{A}_n $ satisfy $d_H\left(\beta(\zeta), \, L \right) < \tfrac{1}{2} n^{-1/4}$, with $\beta$ and $L$ defined as above.
Then, $d(\zeta, \lambda) \leq \frac{5}{2}n^{-1/4}$,
where $d$ is the distance from \eqref{distance} and $\lambda$ is defined by \eqref{lambda}.
\end{lemma}

\begin{proof}
\large
Let $\zeta \in  \mathcal{A}_n $ satisfy $d_H\left(\beta(\zeta), \, L \right) < \tfrac{1}{2} n^{-1/4}$.
We will prove the lemma by defining a specific parametrization for $\zeta$ that satisfies
\begin{equation}\label{sup}
\sup_{t \in [0,1]} d_H\left( \zeta[0,t], \lambda_t \right) \leq \tfrac{5}{2}n^{-1/4}.
\end{equation}
For $j = 1, \cdots, m$, set $t_j = \frac{1}{2} [ \text{Re}(z_j) + \text{Im}(z_j) ]$.
The trace of $\zeta$ will intersect the line segment $l_j$ in several points.  
Let $\zeta(t_j)$ be the point in this intersection that has the largest imaginary part (or equivalently the smallest real part).
Extend $\zeta$ to be continuously defined on $[0,1]$ with $\zeta(0) = 0$ and $\zeta(1) = 1+i$.

Let $\beta_j$ be the branch of $\gamma_\zeta$ that starts at $z_j$ and moves down and to the right until hitting the boundary of $Q$.
Since $\zeta$ cannot cross $\beta_j$, the time $t_j$ is a transition time for $\zeta$, as shown in Figure \ref{Betaj}. In particular,   $\zeta$ will not return to the region below $\beta_j$ after time $t_j$, but $\zeta$ has not yet entered the region above $\beta_j$.
The means $\zeta[0,t_j]$ is within distance $\frac{1}{4n}$ of every vertex of $\mathcal{L}_n$ that is below $\beta_j$, and it does not extend above $\beta_j$.
The triangle $\lambda_{t_j}$ is the closed region below $l_j$ in $\overline{Q}$.
Since $\beta_j$ is within distance $\frac{1}{2} n^{-1/4}$ of $l_j$ by assumption, we have that
 $d_H\left( \zeta[0,t_j], \lambda_{t_j} \right) \leq \frac{1}{2} n^{-1/4}$.

Now let $t \in (t_j, t_{j+1})$.  
Then $\zeta[0,t]$ has completed the region below $\beta_j$ but has not yet entered the region above $\beta_{j+1}$.  Further, recall that the distance between $l_j$ and $ l_{j+1}$ is bounded above by $2n^{-1/4}$. 
Thus,
 $d_H\left( \zeta[0,t], \lambda_{t} \right) \leq \frac{1}{2} n^{-1/4}+2n^{-1/4} = \frac{5}{2}n^{-1/4}$.
 A similar argument applies when $t \in [0,t_1)$ and when $t \in (t_m, 1]$, establishing \eqref{sup}.

\end{proof}

 \begin{figure}
 \centering
\begin{tikzpicture}
\begin{axis}[
xmin=-0.1,xmax=1.1,ymin=0, ymax=1.1,x=2in,y=2in, axis x line*=bottom, axis y line*=left,
axis line style={draw opacity=0}
]
\pgfplotsset{ticks=none}
\addplot[domain=0:0.7, samples=50,smooth,red] {0.67-0.022*sin(deg(50*x))-x};
\addplot[thick, black]
  coordinates{(0,0) (1,0)};
\addplot[thick, black]
  coordinates{(1,0) (1,1)};
  \addplot[soldot] coordinates{(-0.03,0.7)};
\end{axis}
\node[left] at (0.3, 3.5) {$\zeta(t_j)$};
\node[left] at (2.8, 2) {$\beta_j$};
\node[left] at (0, 1.66) {region below $\beta_j$};
\draw[->] (-1.5,1.5) to [out=-90,in=180] (1.5,1);
\node[left] at (5, 4) {region above $\beta_j$};
\end{tikzpicture}
\caption{A sketch of $\beta_j$ (in red), $\zeta(t_j)$, and the regions below and above $\beta_j$. }\label{Betaj}
\end{figure}
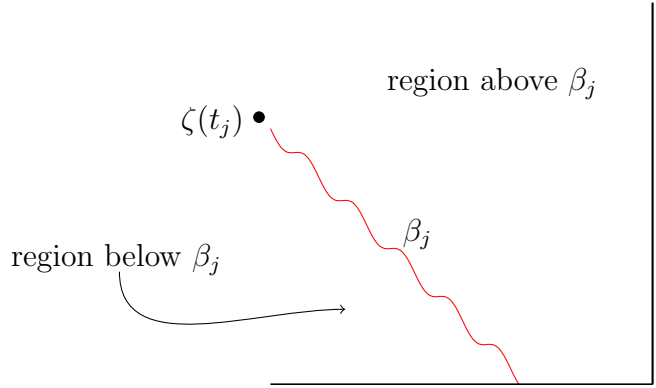

We now have the tools to prove our main result.

\begin{proof}[Proof of Theorem \ref{weakconvergence}]
\large

Our first step is to show that the set $\{\mu_n\}$ is {\it tight}, which means that for any $\epsilon >0$ there exists a compact subset $K_\epsilon$ of $X$ so that $\mu_n(K_\epsilon) \geq 1-\epsilon$ for all $n$.  Once we show this, Prohorov's Theorem (Theorem 83.10 in \cite{rogers_williams_2000}) will imply that $\{\mu_n\}$ is conditionally compact in the space $\cP(X)$ under the topology of weak convergence.  This will show the existence of subsequential limits of $\mu_n$.

Fix $\epsilon >0$ and let $N \in \mathbb{N}$ with $32N^{-1/4} < \epsilon$.
Set $B_n = \big\{ \zeta \in  \mathcal{A}_n  \, : \, d_H\left(\beta(\zeta), \, L_n \right) < \tfrac{1}{2} n^{-1/4} \big\}$ and define
$$K_\epsilon = \left( \bigcup_{n=1}^{N-1} \mathcal{A}_n  \right) \cup \left( \bigcup_{n=N}^\infty B_n \right) \cup \{ \lambda \}.$$
Then  $\mu_n(K_\epsilon) = 1$ for $n < N$, 
and by Lemma \ref{lem:diagonal}, $\mu_n(K_\epsilon) \geq 1- 32n^{-1/4} \geq 1-\epsilon$ for $n\geq N$.
It remains to show that $K_\epsilon$ is compact, by showing that every sequence in $K_\epsilon$ has a subsequence converging to an element of $K_\epsilon$.
A sequence in $K_\epsilon$ must satisfy one of two cases.  In the first case, there is a subsequence contained in 
$ \left( \bigcup_{n=1}^{N-1} \mathcal{A}_n  \right) \cup \left( \bigcup_{n=N}^M B_n \right) \cup \{\lambda\}$ for some $M>N$.
Since this is a finite set, there must be a constant subsequence.
In the second case, there is a subsequence $\zeta_{n_k}$ with $\zeta_{n_k} \in  \bigcup_{n=k}^\infty B_n $.
By Lemma \ref{closetolambda}, $d(\zeta_{n_k}, \lambda) \leq \frac{5}{2} k^{-1/4}$, 
 which shows that $\zeta_{n_k}$ converges to $\lambda \in K_\epsilon$.
 This completes our first step.

Let $\mu$ be the limit of a subsequence $\mu_{n_k}$ under weak convergence.  To finish the proof, we must show that $\mu = \delta_\lambda$.
Set 
$$\tilde{B}_M = \left( \bigcup_{n=M}^\infty B_n \right) \cup \{\lambda\},$$
and note that $\tilde{B}_M $ is a closed set by the reasoning above.
Then by weak convergence and  Lemma \ref{lem:diagonal}
$$\mu(\tilde B_M) \geq \limsup_{k \to \infty} \mu_{n_k}(\tilde B_M) \geq \lim_{k \to \infty} 1-32 n_k^{-1/4}=1,$$
implying that $\mu(\tilde B_M) =1$.  
Since $\tilde B_M$ is a decreasing sequence of sets with limit $\{\lambda \}$, we have that
$\displaystyle \mu(\{\lambda\}) = \lim_{M\to \infty} \mu (\tilde B_M) = 1.$
This shows that $\mu = \delta_\lambda$, completing the proof.
\end{proof}

 \begin{figure}
 \centering
\includegraphics[scale=0.3]{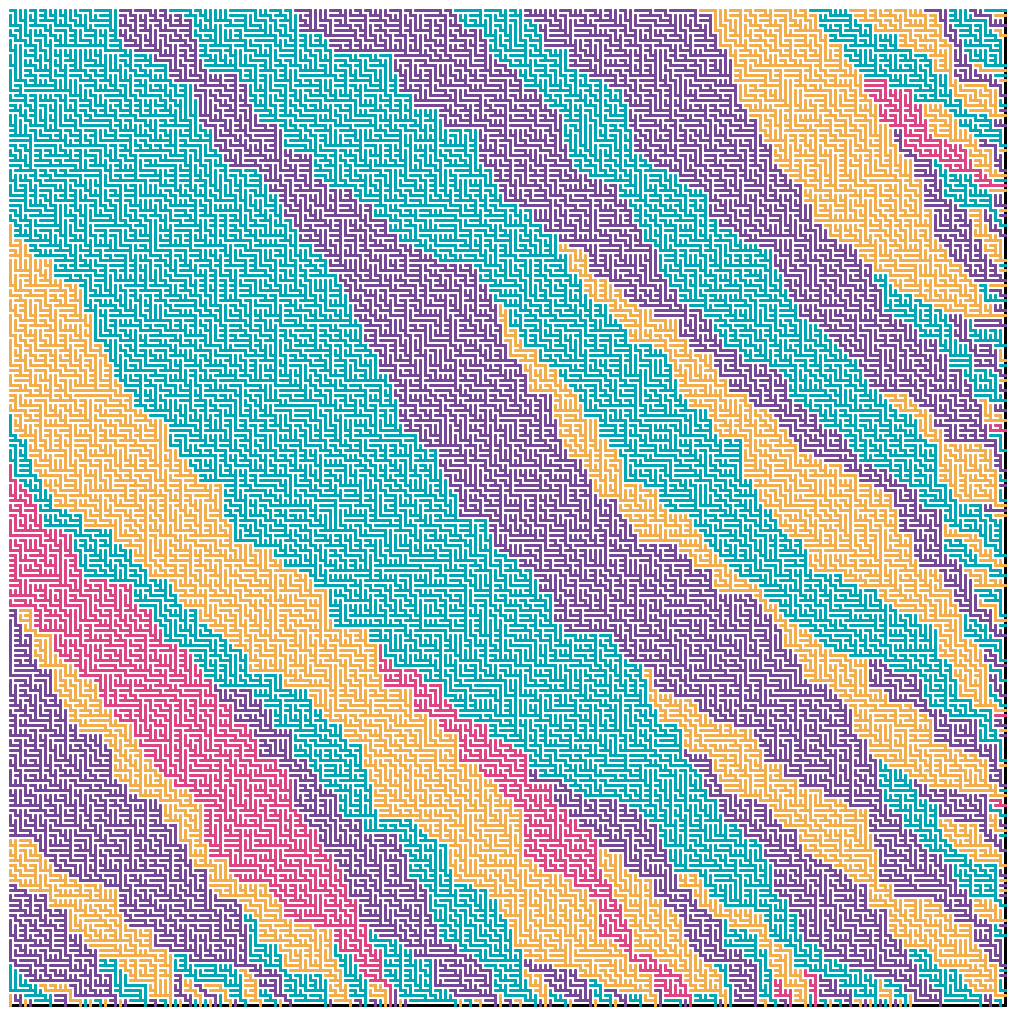}
\caption{A sample of a fair tree when $n=200$, with coloring to show the diagonal nature of the tree branches for large $n$.  Image created by David White.  }\label{fairtree200}
\end{figure}

See Figure \ref{fairtree200} for a simulation of a fair tree with $n=200$.  To create this, we take each vertex $v$ in the interior of $Q$ and select, with equal probability, either the edge leading down from $v$ or the edge leading to the right from $v$.  Note that each edge in the tree has a unique closest vertex in the bottom/right side of $Q$.  To highlight the diagonal nature of fair trees, we have used the same color for edges that share the same closest vertex.

\section{Fair trees for modified grids and for standard grids}\label{sec:fairtrees-grids}
In this section we give more details about the FEU problem and then prove Lemma \ref{lem:fairtree-modgrid}. Also, we show that standard grids might be better suited for this kind of scaling limit.

\subsection{The $1$-densest subgraph}\label{sec:densestsg}
Define the measure of $1$-{\it density} for a graph $G$ as:
\[
\theta(G):=\frac{|E|}{|V|-1}.
\]
Hence, by (\ref{eq:edge-prob}), when $G$ is homogeneous, the edge probabilities  are all equal to $\theta(G)^{-1}$.
Let $\cH(G)$ be the family of all  subgraphs $H\subset G$ such that $H$ is:
\begin{enumerate}
\item[(i)]  non-trivial, meaning that it has at least one edge;
\item[(ii)] connected;
\item[(iii)] vertex-induced, meaning that it contains all the edges of $G$ that have both endpoints in $H$.
\end{enumerate}
Note that we will often identify a subset of edges in $E(G)$ with the corresponding edge-induced subgraph of $G$.
Consider the following {\it $1$-densest subgraph problem}:
Find $H\in\cH(G)$ such that 
\begin{equation}\label{eq:dsp}
\theta(H)=\max_{H'\in\cH(G)}\theta(H').
\end{equation}
In \cite{achpcst}, it was shown that the spanning tree modulus algorithm can provide a solution to this combinatorial problem. Namely, the modulus algorithm computes the optimal density $\eta^*$ in (\ref{prob:feu}). Next, if $H$ is a connected component of the set of edges in $E$ where the minimum value of $\eta^*$ is attained, then $H\in\cH(G)$, $H$ is homogeneous by construction and  $H$ solves the $1$-densest subgraph problem.  Such $H$ is called a {\it homogeneous core} of $G$ 

\subsection{Reducibility and the deflation process}
A proper subgraph $H\in\cH(G)$ has the {\it restriction property} with respect to $G$, if for every fair tree $\ga\in\Ga_G^f$, the restriction $\ga\cap H$ is a spanning tree in $\Ga_H$.
\begin{definition}\label{def:irreducible}
The graph $G$ is said to be {\it reducible}, if it admits a proper subgraph $H\in\cH(G)$ with the restriction property with respect to $G$. A graph that is not reducible is called {\it irreducible}. 
\end{definition}
When the homogeneous core $H$ defined in Section \ref{sec:densestsg} is a proper subgraph of $G$ (and hence $G$ is {\it inhomogeneous}), then $H$ also has the restriction property, see \cite[Theorem 5.2]{achpcst}. Moreover, by \cite[Theorem 5.7]{achpcst}, in this case, not only will the restriction of a fair tree $\ga\in\Ga_G$ to $H$ be a spanning tree of $H$, but it will also be a fair tree of $H$. Secondly, the forest $\ga\setminus H$ is mapped to a fair tree of the quotient graph $G/H$ by the quotient map. And finally, given a fair tree for $H$ and a fair tree for $G/H$, they can be glued together to yield a fair tree of $G$.
This is the essence of the deflation process, namely that arbitrary fair trees of $G$ can be built from fair trees of homogeneous graphs.
We summarize and rephrase the main results of \cite{achpcst} as follow.
\begin{theorem}[\cite{achpcst}]\label{thm:feu-paper}
Let $G=(V,E)$ be a graph  with multi-edges, but no loops. Then, either 
\begin{enumerate}
\item[(i)] $G$ is homogeneous, and $\theta(H)\le\theta(G)$ for all $H\in\cH(G)$; or 
\item[(ii)] $G$ is inhomogeneous, and hence admits a homogeneous {\rm core} $H_0\in\cH(G)$, $H_0\subsetneq G$, that solves the $1$-densest subgraph problem (\ref{eq:dsp}) with $\theta(H_0)>\theta(G)$, and that  has the restriction property with respect to $G$. 

\noindent Moreover, for any fair tree $\ga$ of $G$, the projection of $\ga\setminus H_0$ onto $G/H_0$ is a fair spanning tree of $G/H_0$. And, fair trees of the core $H_0$ can be coupled with fair trees of the quotient graph $G/H_0$ to produce fair trees of $G$.
\end{enumerate}
\end{theorem}

\subsection{Strict denseness}
We say $H\in\cH(G)$ solves the {\it strict denseness problem}, if $H$ is a $1$-densest subgraph
and no proper subgraph is as dense as $H$. In other words, we look for a solution of the $1$-densest subgraph problem that has the following property:
\begin{definition}\label{def:strictd}
A graph $H$ is {\it strictly dense}, if
\[
\theta(H')<\theta(H)
\]
for all proper subgraphs $H'\in\cH(H)$.
\end{definition}
When $H\in\cH(G)$ solves the strict denseness problem, we say that $H$ is a {\it minimal core} for $G$.
Here, we reproduce \cite[Theorem 8.1]{alpc}, which shows that the deflation process described in Theorem \ref{thm:feu-paper} can also be phrased in terms of minimal cores.
Recall that a biconnected graph is a graph that remains connected even when one vertex is removed.
\begin{theorem}[\cite{alpc}]\label{thm:mincore}
Let $G=(V,E)$ be a biconnected graph, possibly  with multi-edges, but no self-loops. Then, either 
\begin{enumerate}
\item[(i)] $G$ is strictly dense, or 
\item[(ii)] $G$ admits a  minimal core $H_0\subsetneq G$. Moreover, in this case, 
\bi
\item[(a)] $H_0$ has the {\rm fair restriction property} with respect to $G$, meaning that for every fair tree $\ga\in\Ga_G^f$, the restriction $\ga\cap H_0$ is a spanning tree in $H_0$;
\item[(b)] for any fair tree $\ga$ of $G$, the projection of $\ga\setminus H_0$ onto the shrunk graph $G/H_0$ is a fair spanning tree of $G/H_0$; 
\item[(c)] fair trees of the core $H_0$ can be coupled with fair trees of the shrunk graph $G/H_0$ to produce fair trees of $G$.
\ei
\end{enumerate}
\end{theorem}
Note that, if $G$ satisfies Theorem \ref{thm:mincore} (ii), then $G$ may or may not be homogeneous.
However, if $G$ is homogeneous and not strictly dense, then it must admit a proper subgraph with the same density.

\subsection{Irreducibility, the all fair property, and strict homogeneity}
Finally, we recall the following definition from \cite{alpc}.
\begin{definition}\label{def:strictly-homogeneous}
A graph $H$ is {\it strictly homogeneous}, if there is a set of edge-weights $\si\in \R_{>0}^{E(H)}$ so that the corresponding $\UST_\si$ pmf $\mu_\si$ is fair. 
\end{definition}
The results in \cite{alpc} show that the fair trees of an arbitrary graph $G$ can be built from the fair trees of strictly homogeneous graphs. In particular, one can use the Aldous-Broder or the Wilson algorithms to sample them. We summarize the results from \cite{alpc} in the next theorem.
\begin{theorem}[\cite{alpc}]\label{thm:main}
Let $G=(V,E)$ be a graph, possibly  with multi-edges, but no loops. Assume $G$ is biconnected. Then, the following are equivalent:
\begin{enumerate}
\item[(a)] $G$ is strictly homogeneous (see Definition \ref{def:strictly-homogeneous});
\item[(b)] $G$ is irreducible (see Definition \ref{def:irreducible});
\item[(c)] $G$ is strictly dense (see Definition \ref{def:strictd});
\item[(d)] $G$ has the all fair property (see Definition \ref{def:allfair}).
\end{enumerate}
Moreover, in this case the weights $\si$ are unique and the corresponding $\UST_\si$ pmf $\mu_\si$ is the fair pmf of maximal entropy among all the fair pmf's.
\end{theorem}
In the next two sections we use this theory to describe the fair trees for modified grids and for regular grids.
\subsection{Fair trees for modified planar grids}\label{sec:modgrids}
We begin by computing the fair trees for the modified grids, hence proving Lemma  \ref{lem:fairtree-modgrid}.
\begin{proof}[Proof of Lemma \ref{lem:fairtree-modgrid}]
\large
To see that $\bar{\Gamma}\subset\Gamma^f_G$, it is enough to show that $\bar{\mu}$ is optimal for  the $\text{FEU}$ problem on $G$. Notice that, when $T\in\bar{\Ga}$,  each edge $e \in E(T)$ lies in  $T \cap E_v$ for some vertex $v\in V(G_0).$ Also, by construction, exactly half of the trees of $ \bar{\Gamma}$  use any given $e$. Therefore, the edge-usage probabilities for $\bar{\mu}$ are:
 \begin{equation*}
 \bP_{\bar{\mu}}(e \in \underline{T})= \sum_{T \in \bar{\Gamma}} \bar{\mu}(T) \mathcal{N}(T,e)=\frac{1}{|\bar{\Gamma}|}\sum_{T \in \bar{ \Gamma}} \mathcal{N}(T,e)=\frac{1}{2},
 \end{equation*}
 where $\mathcal{N}(T,e)=\mathds{1}_{\{e \in T\}}$ gives the usage of $e$ by $T$.
 Coincidentally, this also proves that $G$ is homogeneous. 
  
Before proving the other inclusion, we examine the reducibility of $G$. 
Note that $G$ has $mn+1$ vertices and $2mn$ edges. Therefore, it has $1$-density
\begin{equation}\label{eq:modgrid-denseness}
\theta(G)=2.
\end{equation}  
Let $H$ be the subgraph of $G$ induced by the special vertex $v_0$ and the bottom-right corner of the standard grid $G_0$. Thus, $H$ consists of two nodes connected by a double edge. We refer to such graphs as {\it digons}. In particular, $\theta(H)=\theta(G)=2$. By Theorem \ref{thm:mincore}, $G$ is not strictly dense and hence $G$ is reducible. Moreover, since $G$ is homogeneous,  Theorem \ref{thm:feu-paper} shows that every subgraph in $\cH(G)$ has $1$-density at most $2$. This shows that $H$ is a solution of the strict denseness problem for $G$ and by Theorem \ref{thm:mincore} (ii), we can deflate $G$ with respect to $H$.

If we do that, the resulting quotient graph $G/H$ still maintain the structure whereby there is a special node $v_0$ and the edge-set $E$ can be partitioned as $\cup_{v\in V\setminus\{v_0\}}E_v$ with $|E_v|=2$, as before. In particular, since the edge count decreases by $2$ and the node count decreases by $1$, $G/H$ still has $1$-density equal to $2$. Moreover, this creates additional digons and $G/H$ is again homogeneous, so these digons solve the strict denseness problem, which means that we can repeat the deflation process iteratively.  For instance, we can apply the deflation process going from right to left, row by row, until we are left with a single digon graph. Finally, observe that at each step the two edges that get contracted are exactly the pairs of edges in the sets $E_v$ for $v\in V(G)\setminus\{v_0\}$ defined above. 

Now assume that $T \in \Gamma_G\setminus \bar{\Gamma}$. Then, by definition of $\bar{\Gamma}$, and the Pigeonhole principle, there exists vertices $v^*,v_*\in V(T)$, such that $|T\cap E_{v^*}|=2$ and $|T\cap E_{v_*}|=0$. Indeed, $T$ must choose $mn$ edges from $E$, which is partitioned into $mn$ pairs of edges of the form $E_v$. 
 However, this means that in the deflation process described above, we are guaranteed to get to a point where the two edges in, say, $E_{v_*}$ connect $v_*$ to a single other node $u_*$, so as to form a subgraph $H_*$ of $1$-density two in a larger homogeneous graph $G_*$ of the same $1$-density. Let $T_*$ be the tree obtained from $T$ by removing all the pairs $E_v$ that have been involved in the deflation thus far. Then, $T_*$ does not restrict to a spanning tree of $H_*$, and thus $T$ is forbidden. 
\end{proof}

\subsection{Fair trees for regular planar grids}\label{sec:grids}

One reason the scaling limit of the Peano paths in the case of modified grids turned out to be deterministic is because the family of fair spanning trees is much smaller than the family of all spanning trees. Indeed, for an $n$-by-$n$ grid the number of spanning trees is asymptotically $e^{\frac{4C}{\pi}n^2}$ where $C$ is Catalan's constant \cite{wu1977}, i.e., they grow like $a^{n^2}$, for $a=3.2099\cdots$. On the other hand, the number of fair trees on the modified  $n$-by-$n$ grid grows like $2^{n^2}$.
Therefore, to explore a setting that may have a richer collection of fair trees, we return to regular planar grids.

\begin{lemma}\label{lem:planar-grids}
 The $m$-by-$k$ planar grid $G_{m,k}$ is strictly dense.
\end{lemma}
In particular, by Lemma \ref{lem:planar-grids}, Theorem \ref{thm:main}, and Definition \ref{def:strictly-homogeneous}, 
there is a unique set of edge weights $\si>0$ on $E(G_{m,k})$ so that the corresponding $\UST_\si$ pmf $\mu_\si$ is fair. This suggests a different approach to the scaling limit of fair Peano paths. Namely, first sample fair $\UST_\si$ trees from the grid $G_{m,k}$ and then condition such trees to  contain the bottom and right-hand side of the grid. This should give rise to a much richer family of spanning trees that may have a non-deterministic scaling limit. We will pursue these considerations in future work.
\begin{proof}[Proof of Lemma \ref{lem:planar-grids}]
\large
First note that:
\[
|V(G_{m,k})|=mk\qquad\text{and}\qquad |E(G_{m,k})|=(m-1)k+m(k-1)=2mk-m-k
\]
So
\[
\theta(G_{m,k})=\frac{2mk-m-k}{mk-1}=2-\frac{m+k+2}{mk-1}
\]
In particular,
\[
\frac{\partial \theta(G_{m,k})}{\partial m}=-\frac{(mk-1)-k(m+k+2)}{(mk-1)^2}=\frac{(k+1)^2}{(mk-1)^2}>0,
\]
and by symmetry the same holds for the derivative in $k$, meaning that the $1$-density of rectangular grids is strictly increasing with the size of the grid, is always less than $2$, and in fact tends to $2$ as $m,k\rightarrow\infty$.

Now let $H$ be a proper connected vertex-induced subgraph of $G=G_{m,k}$. We first consider the smallest rectangular grid containing $H$. Namely, let
\[
i_{min}:=\min\{i:\ \exists j, (i,j)\in V(H)\}
\]
and similarly define $i_{max}, j_{min}$, and $j_{max}$. Then, define the translated grid
\[
\tilde{G}:=(i_{min},j_{min})+G_{i_{max}-i_{min},j_{max}-j_{min}}.
\]
If $H=\tilde{G}$, then $\theta(H)=\theta(\tilde{G})< \theta(G)$ and we are done. If $H$ is a proper subgraph of $\tilde{G}$, then since $\theta(\tilde{G})\le \theta(G)$, it will be enough to show that $\theta(H)<\theta(\tilde{G})$. In particular, without loss of generality we can assume that $\tilde{G}=G_{m,k}$.

As an embedded planar graph $G_{m,k}$ has $(m-1)(k-1)$ bounded faces. Whenever such a face has all four vertices in $V(H)$ we fill it in so that $H$ becomes a connected compact set $\hat{H}$ in the plane. In particular, the complement of $\hat{H}$ is an open set consisting of finitely many bounded components $\Om_j$, $j=1,\dots,p$, and one unbounded component $\Om_0$. The boundary of each one of these components can be parametrized by a curve $\Ga_j$, $j=0,\dots,p$. This can be seen by thickening $\hat{H}$ a little bit, for instance by defining
\[
\hat{H}_\ep:=\bigcup_{z\in\hat{H}}\{w\in\C: |w-z|\le \ep\}.
\]
For each $\ep$, the boundary of $\hat{H}_\ep$ consists of $C^1$-smooth Jordan curves $\Ga_{j,\ep}$, $j=0,\dots,p$. If we parametrize each $\Ga_{j,\ep}$ with its arc-length parametrization, then, as $\ep$ tends to $0$, they converge uniformly to the curves $\Ga_j$.
 
We begin by looking at the boundary of the unbounded component $\Om_0$ and think of it as being parametrized counter-clockwise. In particular, $\Ga_0$ can be taken to be piecewise linear so that the derivative $\Ga_0^\prime$ is well defined away from the nodes of $\Z^2$.
Since $\Ga_{0,\ep}$ bounds the simply connected domain $\Om_{0,\ep}\cup \{\infty\}$, by the argument principle, the change in argument for the derivative $\Ga_{0,\ep}^\prime$, as the curve completes one full loop, is $2\pi$. As $\ep$ tends to $0$, this property is inherited by $\Ga_0$, as long as the argument of the derivative $\Ga_0^\prime$ is properly defined.
At any moment when $\Ga_0$ is not at a node of $\Z^2$, the right hand of a walker traveling along with $\Ga_0$ will always be touching $\Om_0$. In particular, the only changes in the argument of $\Ga_0^\prime$ that are allowed when $\Ga_0$ passes through a node of $\Z^2$ are 
\[
0,\frac{\pi}{2},\pi,\text{ and, }-\frac{\pi}{2}.
\]
In words, either the walker goes straight, turns $90^\circ$ left, does a $180^\circ$ turn in the positive direction, or turns $90^\circ$ right. This can be verified by looking at the unbounded component $\Om_{0,\ep}$ of the thickened $\hat{H}_\ep$. Indeed, if $\Ga_0$ were to do a $180^\circ$ turn in the negative direction, then while walking along $\Ga_{0,\ep}$ the right hand would be touching the thickened neighborhood of an edge, and this neighborhood would disappear as $\ep$ tends to $0$, so that edge would not be part of the boundary of $\Om_0$, which leads to a contradiction. 

Now assume $v\in \Z^2$ is a node where $\Ga_0$ makes a right turn ($90^\circ$  in the negative direction), and let $u\in\Z^2$ be the node visited by $\Ga_0$ just before $v$, and $w\in\Z^2$ the one visited just after $v$.
Then, $u,v,w$ bound a face $f$ of $G_{m,k}$ and the fourth corner $t$ does not belong to $H$. To see why, assume by contradiction that all four corners are in $V(H)$, then the face would be contained in $\hat{H}$. But then the right hand would not be touching the unbounded component $\Om_0$ in this case.
Also, note that since $H$ is vertex-induced the two sides of $f$ incident at $t$ also do not belong to $E(H)$. 

Now, if we add $t$ to $V(H)$, to get a new vertex-induced graph $H^\prime$. Since one of the coordinates of $t$ is equal to one of the coordinates of either $u$ or $w$, we see that $H^\prime$ is still a subgraph of $G_{m,k}$. Moreover, we are guaranteed that 
\[
E(H^\prime)\ge E(H)+2.
\] 
In particular,
\[
\theta(H^\prime)= \frac{E(H^\prime)}{V(H^\prime)-1}\ge\frac{E(H)+2}{V(H)}
\]
Since $H$ is a subgraph of the grid $G_{m,k}$, we always have $\deg_H(x)\le 4$. Moreover,  every node $x\in\Ga_0$ has $\deg_H(x)< 4$ and we can assume that $\Ga_0$ visits at least four distinct nodes (the remaining cases are when $H$ consists of a single edge, or a path with two edges). Thus, we can assume that there are at least four vertices of $H$ with $\deg_H(x)<4$, and, by the Handshake Lemma,
\[
2E(H)=\sum_{x\in V(H)}\deg_H(x)<4 V(H)-4.
\]
This implies that $\theta(H')>\theta(H)$ because
\begin{align*}
 \frac{E(H)+2}{V(H)} >\frac{E(H)}{V(H)-1}
& \Longleftrightarrow (E(H)+2)(V(H)-1)> E(H)V(H)\\
& \Longleftrightarrow 2V(H)-2> E(H).
\end{align*}
If $H^\prime=G_{m,k}$, we are done. Otherwise, we can repeat the same argument with $H$ replaced by $H^\prime$. This process has to end, so without loss of generality, we can assume that $\Ga_0$ never makes any right turns. In particular, the argument of $\Ga_0^\prime$ can only change by $0$, $\pi/2$, or $\pi$. So since there has to be an even number of argument changes that are positive, and the sum must be $2\pi$, we either get two changes by $\pi$ or four changes by $\pi/2$. In either case, we see that $\Ga_0$ describes the boundary of  a rectangular grid, and by assumption this grid has to be $G_{m,k}$.

Finally, consider a bounded component $\Om_j$ with $j\ge 1$. Assume that the boundary of $\Om_j$ is described in the clockwise direction by a curve $\Ga_j$, constructed as $\Ga_0$ above, using the thickened set $\hat{H}_\ep$. Once again the allowed argument changes at a grid node are $0,\pi/2,\pi$, and $-\pi/2$. Also, when walking along $\Ga_j$ the right-hand touches $\Om_j$ and the total argument change must equal $-2\pi$. In particular, there must always be at least four right turns with argument change $-\pi/2$. By repeating the argument above we see that each right turn identifies a vertex in $\Om_j$ that can be added to $H$ in a way to increase the $1$-density, and therefore after finitely many steps we get that the component $\Om_j$ has been filled in. Hence, after finitely many steps, $H=G_{m,k}$. This proves the lemma.
\end{proof}

\section{Appendix}

For the convenience of the reader we include a proof of needed result about the running maximum $M_k$ of $|S_k|$ for the 1-dimensional simple random walk $S_k$.

\begin{lemma}\label{runningmax}
$\mathbb{P}(M_k \geq r) \leq 2k/r^2$ for $r >0$.
\end{lemma}

\begin{proof}
\large
Let $r>0$ and assume that $r\in \mathbb{Z}$.
Set $M_k^+ = \max\{ S_0, S_1, \cdots, S_k\}$ and $M_k^- = \min \{ S_0, S_1, \cdots, S_k\}$.  
Then for $v \in \mathbb{Z}$,
$$ \mathbb{P}(M_k^+ \geq r \text{ and } S_k = v) = 
\begin{cases} \mathbb{P}(S_k = v) \;\;\;\;\;\;\;\;\;\; \text{ if } v\geq r\\
\mathbb{P}(S_k = 2r-v) \;\; \text{ if } v < r
\end{cases}.
$$
(Reason for second case:   The probability of going up to $r$ and then down to $v$ is the same as going up to $r$ and continuing up to $2r-v$.  In other words, think of reflecting the future of the simple random walk over the line $y=r$ at the point it first reaches height $r$.)
Now summing over $v$ gives
\begin{align*}
 \mathbb{P}(M_k^+ \geq r) &= \sum_{v=r}^\infty \mathbb{P} (S_k = v) +  \sum_{v=-\infty}^{r-1} \mathbb{P} (S_k = 2r-v) \\
   &= \mathbb{P}( S_k = r)+2 \sum_{v=r+1}^\infty \mathbb{P} (S_k = v) \\
   &= \mathbb{P}( S_k = r)+2\mathbb{P}( S_k \geq r+1) \\
   &\leq 2\mathbb{P}( S_k \geq r)\\
   &=\mathbb{P}( |S_k| \geq r).
\end{align*}
Therefore
\begin{align*}
 \mathbb{P}(M_k \geq r) &\leq  \mathbb{P}(M_k^+ \geq r) +  \mathbb{P}(M_k^- \leq -r) \\
 &= 2 \mathbb{P}(M_k^+ \geq r) \\
 &\leq 2\mathbb{P}( |S_k| \geq r) \\ 
 &\leq 2k/r^2,
\end{align*}
where the last step is from Chebyshev's Inequality.
This proves the result for integer $r$, but since $M_k$ is integer-valued, we are able to extend this to all $r>0$.
\end{proof}

\bibliographystyle{acm}
\bibliography{pmodulus}
\def\cprime{$'$}


\end{document}